\documentclass{article} 

\usepackage{amsmath,amsthm}     
\usepackage{graphicx}     
\usepackage{hyperref} 
\usepackage{url}
\usepackage{amsfonts} 

\usepackage{multicol}

\theoremstyle{theorem}
\newtheorem{theorem}{Theorem}

\theoremstyle{definition}
\newtheorem*{definition}{Definition}

\allowdisplaybreaks

\makeatletter
\@addtoreset{footnote}{page}
\makeatother

\begin{document}

\title{The Warden's de Bruijn Sequence}

\author{Joseph DiMuro\\               %%%% Leave ALL of these as is in your initial submission
Biola University\\    %%%% to allow for double blind reviewing.
                %%%% They should be filled in when you are submitting
joseph.dimuro@biola.edu}                      %%%% your final manuscript.

\date{}

\maketitle

\noindent  Imagine the following scenario: you are in prison, facing a long prison sentence. This particular prison has an unusual warden; for his own amusement, he offers his prisoners a chance at freedom. He shows you a row of coins on a table; some are showing heads, and some are showing tails.

\begin{center}
  \includegraphics[height=3cm]{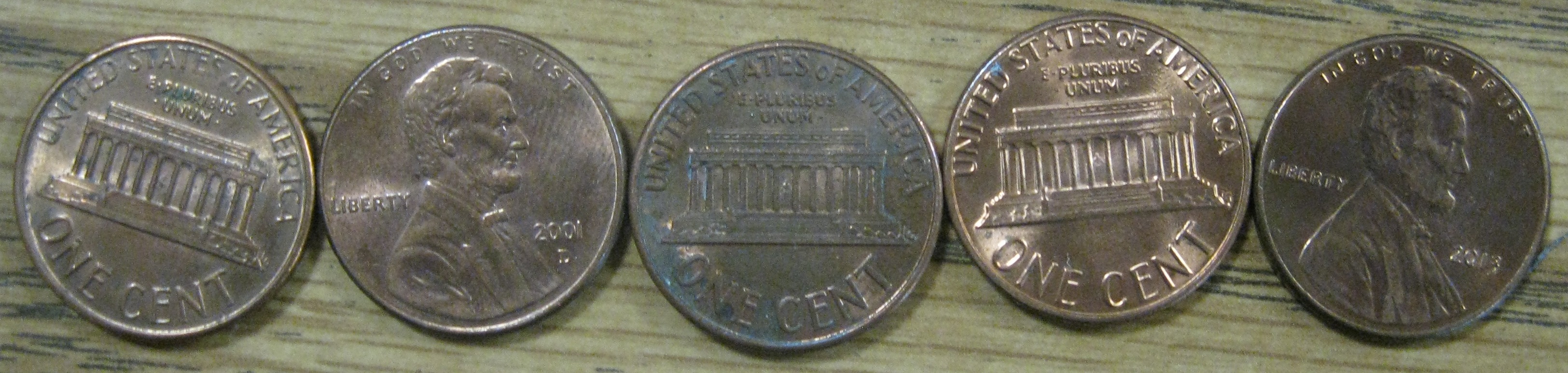}
\end{center}

You and the warden are going to play a game. Your goal is for all of the coins to show tails; if you do that, you've earned your freedom. This game will be played at a rate of one move per day, so of course, you want to win in as few moves as possible. The warden will do his best to delay your victory for as long as he can; indefinitely, if possible.

Each day of your imprisonment, either you or the warden will transfer the rightmost coin to the far left of the row. If the coin is showing heads, then you transfer the coin, and you may optionally flip the coin to tails. But if the coin is showing tails, then the warden transfers the coin, and he may optionally flip the coin to heads.

For instance, starting from THTTH, you get to transfer the rightmost coin, and you may flip it if you wish. Let's say you flip the coin to tails; after the transfer, the position is TTHTT. The warden then gets to transfer the next coin; perhaps he flips the coin to heads, producing HTTHT. The next coin is tails, so the warden takes another turn; perhaps he leaves the coin as tails, producing THTTH. It is now your turn again; sadly, we are back to the starting position, so you have not made any progress.

In this game, you can flip heads to tails, but the warden can flip tails back to heads just as quickly. So what can you do? Can you earn your freedom, and if so, how do you do so in the fewest number of moves? Or can the warden keep you in prison indefinitely? The reader is encouraged to give this game a try before reading on. 

\section{The Warden Loses}

It turns out that, from any starting position, you can ensure your freedom in a finite amount of time if you play correctly.

\begin{theorem}[Basic Win Strategy]
Starting from any sequence of heads and tails, there is a strategy for the prisoner that ensures that all coins are flipped to tails in finitely many moves.
\end{theorem}

\begin{proof}
We can represent any possible position as a number in binary, using a 0 to represent heads and a 1 to represent tails. For example, the sequence THTTH would be represented as 10110, which is 22 in base ten.
  
Assume there are $n$ coins on the table. After $n$ moves, each coin will be back in its original position (though possibly flipped). We will show that the prisoner has a strategy that ensures that, after $n$ moves, either the coins will all be showing 1's, or the coins will be showing a smaller binary number than before. But these binary numbers can't decrease forever, so eventually, the prisoner will get all 1's and win.
  
Here's the strategy for the prisoner: each round of $n$ moves, flip all 0's into 1's, until the warden flips a 1 into a 0. (If he never does, then all the coins will be showing 1, and the prisoner wins.) Thereafter, leave all 0's as 0's. If the prisoner follows this strategy, then each cycle of $n$ moves, the leftmost coin that is changed will be a coin that the warden flipped from a 1 to a 0, thus producing a smaller binary number.

For example: say we have the position 0101001, which is the binary representation of 41. Say the warden transfers the rightmost 1 without flipping it, producing $1|010100$. You then try for the win, flipping the next two 0's into 1's; that yields $111|0101$. The warden might then spoil your chance for a win by flipping the next 1 into a 0, producing $0111|010$. You then stop flipping 0's into 1's; three moves later, the position is 0100111, assuming the warden doesn't flip the next 1 into a 0. This is the binary representation of 39, which is smaller than the original number. You then follow the same strategy again to produce an even smaller binary number; this process cannot continue indefinitely, so you will eventually get all 1's and win.  
\end{proof}

Using this strategy, you can ensure that you gain your freedom after $n 2^n$ moves; the game can last for no more than $2^n$ cycles, and each cycle consists of $n$ moves.

However, this is not the best that can be done. Consider the starting sequence 01110, which is the binary representation of 14. If you blindly follow the basic strategy, it could take as many as 75 moves for you to escape: 14 sets of 5 moves to reduce the position to 00000, and then 5 more moves to turn the position into 11111. But if you simply transfer the rightmost 0 without flipping it, the position will then be 00111, the binary representation for 7. You can now start following the basic strategy, and win in at most 40 more moves. Clearly, the basic strategy can be greatly improved.

So, what are the optimal strategies for you and the warden? We will get to that eventually. First, let's generalize this game.

\section{The Warden Takes to Dice}

The warden got tired of playing this game with coins, so to spice things up, he decided to play with dice instead. He came up with a variation which uses $m$-sided dice, for any positive integer $m$.

The sides of each die contain the integers from 0 to $m-1$. The goal for the prisoner is to get the dice to all show the highest possible value, $m-1$. For example, if $m=6$, then the numbers on each die range from 0 to 5, instead of the traditional 1 to 6.

\begin{center}
  \includegraphics[height=3cm]{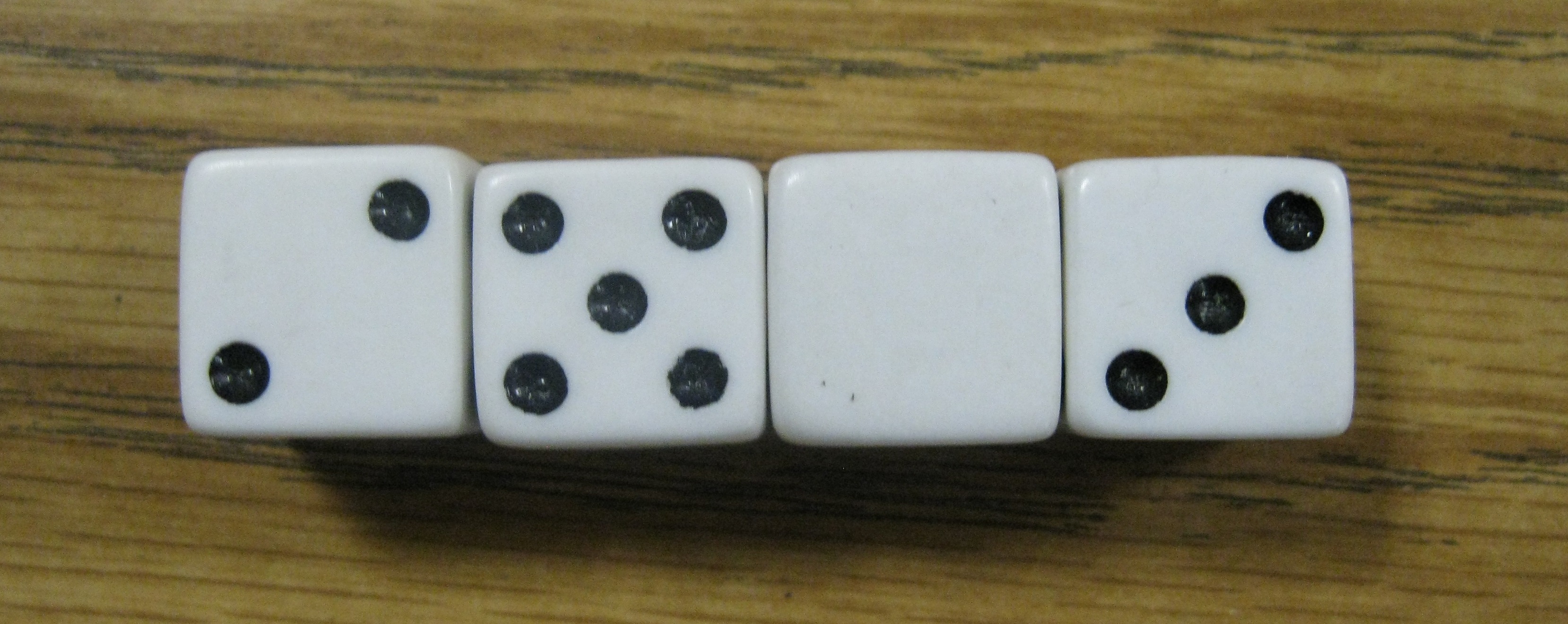}
\end{center}

The warden needed some time to come up with interesting rules for this variation. He wanted to set things up so that, the higher the number on the die to be transferred, the more control the warden would have over what number the die would be changed to. The rule he decided on: each turn, the warden has the option of transferring the rightmost die to the far left, and decreasing the value on that die. If the warden doesn't want to do this (or can't, because the die is showing 0), then the warden may pass; then the prisoner must transfer the rightmost die, and has the option of increasing the value on that die.

For example, starting from the position 2503, the warden may transfer the 3 himself, changing the value on that die to a 0, 1, or 2. Or the warden may pass; the prisoner then must transfer the 3, and may either leave it as a 3 or change it to a 4 or a 5.

Let's say the warden transfers the 3 himself, changing it to a 2. The position then is 2250. Now the warden cannot lower the value on the rightmost die, so he has no choice but to pass; the prisoner then transfers the 0, and can change the value on that die to anything he wishes. And so on.

Note that this rule generalizes the rule for the coin-flipping case. Each coin can show either a 0 (heads) or a 1 (tails). If a coin to be transferred shows a 1, then the warden can either change that coin to a 0, or he can pass... but then the prisoner would be forced to leave that 1 as a 1. So if the rightmost coin shows a 1, the warden has full control. But if a coin to be transferred shows a 0, then the warden is forced to pass, which means the prisoner has full control.

Under this ruleset, we can easily show:

\begin{theorem}[Basic Win Strategy, Revisited]
  If the game is played with $n$ $m$-sided dice, then regardless of the starting setup, there is a strategy for the prisoner that ensures that all dice are showing the largest value ($m-1$) in a finite number of moves.
\end{theorem}

The proof is very similar to the proof for the game with coins. The set of dice can be treated as showing a number in base $m$. Each cycle of $n$ moves, the prisoner changes every die to $m-1$ until the warden decreases the number on a die. For the rest of the cycle, the prisoner leaves all dice alone. Following this strategy, the dice will either be reset to all $(m-1)$'s, or will show a smaller base $m$ number than before. This process can last only a finite number of cycles, so the prisoner wins in a finite number of moves.

But how can we find a strategy for the prisoner to follow that guarantees freedom in as few moves as possible? To find the pattern, let's work through the case where $m=3$ and $n=3$. We will try to find the \textit{remoteness} of each of the 27 positions; the remoteness is the number of moves the game will last from that position, assuming the winning player (the prisoner) tries to win as quickly as possible, and the losing player (the warden) tries to delay defeat for as long as possible. The term \textit{remoteness} was first used in this way by Cedric Smith (\cite{CABS}, p. 56).

First, some notation: for the rest of the paper, we will use Greek letters ($\alpha$, $\beta$, etc.) to refer to positions in the game. For example: if $n=m=3$, one possible position is $\alpha=102$. Given any position $\alpha$, we will use the notation $r(\alpha)$ for the remoteness of $\alpha$.

Note: if $n=m=3$, the only position with remoteness 0 is the goal position, 222. But we can also consider 222 as a starting position, and say that the prisoner wins once 222 is reached once again. As a starting position, 222 has nonzero remoteness; we will shortly determine what that remoteness is. In general, if $\alpha$ is the goal position (consisting of all $m-1$'s), we will write $r(\alpha)$ to denote the remoteness of $\alpha$ treated as a starting position. (So, $r(\alpha)>0$.)

If $n=m=3$, which positions have remoteness 1? From any such position, it must be possible to move to a position of remoteness 0 (namely, 222) in just one move. So any position of remoteness 1 must have the form 22X. But the warden can prevent a win in one move from either 221 or 222; he can move to 022. So 220 is the only position of remoteness 1; the warden is forced to pass, and the prisoner can move to 222.

Which positions have remoteness 2? Any such position must have the form 20X, so that it is possible to move to a position of remoteness 1 (namely, 220) in just one move. But again, the warden can move to something other than 220 if the position is 201 or 202. So the only position of remoteness 2 is 200. Similarly, the only position of remoteness 3 is 000.

But what about remoteness 4? Any such position must be of the form 00X, but not 000, because that position has remoteness 3. From 002, the warden can avoid moving to 000; he can move to 100 instead. But from 001, the warden's best option is to move to 000; if the warden passes instead, then the prisoner can move to 200, which has a smaller remoteness than 000. So indeed, 001 has remoteness 4, and it is the only position of remoteness 4.

Continuing this reasoning, we can find the remoteness of every possible position. It turns out that there is exactly one position of each remoteness, from 0 to 27. The complete list is as follows, starting from remoteness 0:

\begin{center}
222, 220, 200, 000, 001, 010, 100, 002, 020, 201, 011, 110, 101, 012,

120, 202, 021, 210, 102, 022, 221, 211, 111, 112, 121, 212, 122, 222 \,
\end{center}

Is it just a coincidence that we never have two positions with the same remoteness? Not at all; we will prove that this always happens, regardless of the number of dice or the number of sides on each die.

First of all, notice what happens when compare the remoteness of two positions that are identical except for the last digit. For instance, we can compare the positions 100, 101, and 102. Notice that $r(100)<r(101)<r(102)$. This actually holds true throughout the sequence; if two positions are identical except for the last digit, the one with the smaller final digit has the smaller remoteness. (This rule assumes that we treat the goal state as a starting position, not as an ending position. Here, the goal position is 222, and we have $r(220)<r(221)<r(222)$.)

\begin{theorem}[Larger Numbers Help the Warden]\label{largerhelpswarden}
  Consider the warden's game played with $n$ dice, each with $m$ sides. Consider two possible positions $\alpha a$ and $\alpha b$, where $\alpha$ is an arbitrary string of $n-1$ $m$-ary digits. Then if $a<b$, then $r(\alpha a)\le r(\alpha b)$.
\end{theorem}

The reason: increasing the value on the rightmost die can only help the warden, because it gives the warden more options. If the rightmost die is showing $a$, then the warden can set that die to any value from 0 to $a-1$, or he can pass and let the prisoner choose any value $a$ or more. If the rightmost die is showing $b>a$, then the warden has the additional options of setting that die to a value from $a$ to $b-1$, or of passing and letting the prisoner choose a value $b$ or more. Those extra options can't hurt the warden, so the warden can keep the game going at least as long with a $b$ as the final digit as with an $a$ as the final digit.

However, this does not prove that $\alpha a$ will have a \textit{strictly smaller} remoteness than $\alpha b$. Could it be that there exist positions $\alpha a$ and $\alpha b$ (with $a<b$) with the same remoteness, so that the warden's extra options from $\alpha b$ don't help the warden? The next theorem proves that this can't happen.

\begin{theorem}[The Game Tree is a Single Chain] Consider the warden's game played with $n$ dice, each with $m$ sides. Then there is exactly one position of each remoteness from 0 to $m^n$, and the position with the largest remoteness $m^n$ is the position of all $m-1$'s.
\end{theorem}

\begin{proof}
  First, consider a positive integer $x$ such that there are no positions of remoteness $x$. Then there won't be any positions of remoteness $x+1$ either; from a position of remoteness $x+1$, the next move would have to be to a position of remoteness $x$ (of which there are none), assuming both players play properly. Continuing this, there won't be any positions of remoteness $x+2$, and so on. So if $x$ is the smallest positive integer where there are no positions of remoteness $x$, then there won't be any positions of remoteness $y>x$.
  
  Second, we need to show that for each positive integer $x$, if there is a position of remoteness $x$, then there is only one such position. This, coupled with the paragraph above, will show that there is exactly one position of each remoteness from 0 to $m^n$.
  
  Assume, to the contrary, that there are two positions with the same remoteness. Let $x$ be the smallest positive integer such that there are two positions of remoteness $x$. Then there is exactly one position of each remoteness from 0 to $x-1$. Let $c\alpha$ be the position of remoteness $x-1$, where $\alpha$ is a string of length $n-1$. Then the two positions of remoteness $x$ have the form $\alpha a$ and $\alpha b$, where $a<b$. We must in fact have $b=a+1$, and the positions $\alpha 0,\alpha 1,\ldots \alpha(a-1)$ all have smaller remoteness than $\alpha a$.
  
  Consider the optimal moves from the positions $\alpha 0,\alpha 1,\ldots \alpha a$; let's call the resulting positions $c_0\alpha,c_1\alpha,\ldots c_a\alpha=c\alpha$. Since $\alpha(a+1)$ has the same remoteness as $\alpha a$, the optimal move from $\alpha(a+1)$ should be the same as from $\alpha a$: namely, to $c\alpha$. But that is not the case; the warden has a way to avoid moving from $\alpha(a+1)$ to any of the positions $c_0\alpha,c_1\alpha,\ldots c_a\alpha$. Namely:
  
  \begin{itemize}
    \item
    If there is a nonnegative integer $s\le a$ that is not on the list of $c_0,c_1,\ldots c_a$, then the warden can move to $s\alpha$.
    \item
    If the numbers $c_0,c_1,\ldots c_a$ are exactly the numbers from 0 to $a$ in some order, then the warden may pass, forcing the prisoner to move to a position other than $c_0\alpha,c_1\alpha,\ldots c_a\alpha$.
  \end{itemize}
  
  Thus, from $\alpha(a+1)$, the warden can make sure that the next move is not to a position of remoteness at most $x-1$. So, $\alpha(a+1)$ is not a position of remoteness $x$, contradiction. Thus, there cannot be two positions of the same remoteness; there is exactly one position of each remoteness from 0 to $m^n$.
  
  Lastly, we need to show that if $\alpha$ is the position of all $m-1$'s, then $\alpha$ (treated as a starting position) is the position of greatest remoteness. This is true because, if there was a position $\beta$ where $r(\beta)=1+r(\alpha)$, then the optimal move from $\beta$ would be to $\alpha$. But that would be a \textbf{winning} move for the prisoner, and we would have $r(\beta)=1$, not $r(\beta)=1+r(\alpha)>1$. So this cannot happen; $\alpha$ (as a starting position) must be the position of greatest remoteness.
\end{proof}

Let's now go back to game ``tree'' for the game with $n=3$ dice, each with $m=3$ sides:

\begin{center}
222, 220, 200, 000, 001, 010, 100, 002, 020, 201, 011, 110, 101, 012,

120, 202, 021, 210, 102, 022, 221, 211, 111, 112, 121, 212, 122, 222 \,
\end{center}

There's a lot of redundant information in this list, because the last $n-1$ digits of each position match the first $n-1$ digits of the next position. To save space, we can just write down the new digit obtained at each step. We thus get the following:

\begin{center}
222000100201101202102211121222
\end{center}

But we can do a bit better here. Remember, the goal position and the position with largest remoteness are the same: 222. Because of that, we can think of this string as a loop of digits; the sequence starts with 000, and continues to the 222, but then loops back around to 000. I'll write the sequence like this:

\begin{center}
(222)000100201101202102211121222
\end{center}

This is a loop of 27 digits, not a string of 30 digits; the 222 is repeated for visual clarity (otherwise, it is difficult to see where 220 and 200 occur in the loop). The entire optimal strategy for the game is summarized in this loop of digits; for each turn of the game, the optimal move is to shift one step to the left in this sequence. Because the prisoner can win from every position, every possible combination will appear exactly once in this sequence. Thus, this is a \textit{de Bruijn sequence}.

\begin{definition}
  Given positive integers $m$ and $n$, an \textbf{$m$-ary de Bruijn sequence of rank $n$} is a sequence of $m^n$ numbers from the set $\{0,\ldots,m-1\}$, such that each possible string of $n$ numbers appears exactly once in the sequence (treated as a loop).
\end{definition}

As a further example, here's the de Bruijn sequence we would obtain from the game with four coins; it's a 2-ary de Bruijn sequence of rank 4.

\begin{center}
(1111)0000100110101111
\end{center}

In general, there are many $m$-ary de Bruijn sequences of rank $n$. Here are a few more 2-ary de Bruijn sequences of rank 4 (these are unrelated to the warden's game):

\begin{center}
(1011)0000100111101011

(1101)0000101100111101

(1111)0000101101001111
\end{center}

There are actually sixteen different 2-ary de Bruijn sequences of rank 4. Each of the other twelve can be obtained from one of these four, by reversing the sequence and/or replacing 0s with 1s and vice versa (\cite{dB}, p. 758).

Note: since de Bruijn sequences are loops of digits, it is possible to choose any digit in such a sequence to be the ``first digit''. I have intentionally written these four sequences so that they all start with four 0's, ignoring the parenthesized digits (which are there just for visual clarity). Given that, which one of these four sequences represents the smallest binary number? It's the one that comes from the warden's game. We say that the warden's de Bruijn sequence is ``lexicographically smaller'' than the others; I've listed the four sequences ``alphabetically'' above (where the ``alphabet'' is $\{0,1\}$), and the warden's de Bruijn sequence is first out of the four.

In fact, we could check all sixteen possible de Bruijn sequences, and the warden's de Bruijn sequence would be lexicographically smaller than all the others. And this is to be expected; when we construct the warden's de Bruijn sequence, we always choose the smallest digit we can at each step (from Theorem \ref{largerhelpswarden}). Thus:

\begin{theorem}
  Given positive integers $m$ and $n$, the sequence representing the game tree for the game with $n$ $m$-sided dice is the \textbf{lexicographically minimal} $m$-ary de Bruijn sequence of rank $n$.
\end{theorem}

Note: this result, in the case where $m=2$, was first proven by Gera Weiss (\cite{GW}, p. 4645).

Lexicographically minimal de Bruijn sequences are famous mathematical structures; they are sometimes called ``granddaddy'' de Bruijn sequences. Various methods have been given for constructing them. There are fast algorithms for finding the location of a particular string in a lexicographically minimal de Bruijn sequence (\cite{KRR}), but for some other de Bruijn sequences, there are algorithms that are even faster (\cite{MEP},\cite{JT}).

\section{The Warden Changes the Goal}

The warden was moderately satisfied with his game, but he was getting tired of having to use so many different kinds of dice. (Also, it was hard to find dice with certain numbers of sides; for example, 5-sided dice and 7-sided dice are pretty hard to come by.) So he decided to have a whiteboard installed in his office. This made the game much easier to play; the warden would write a starting set of numbers on the board, and the players would take turns erasing the rightmost number and writing a new number on the far left.

\begin{center}
  \includegraphics[height=3cm]{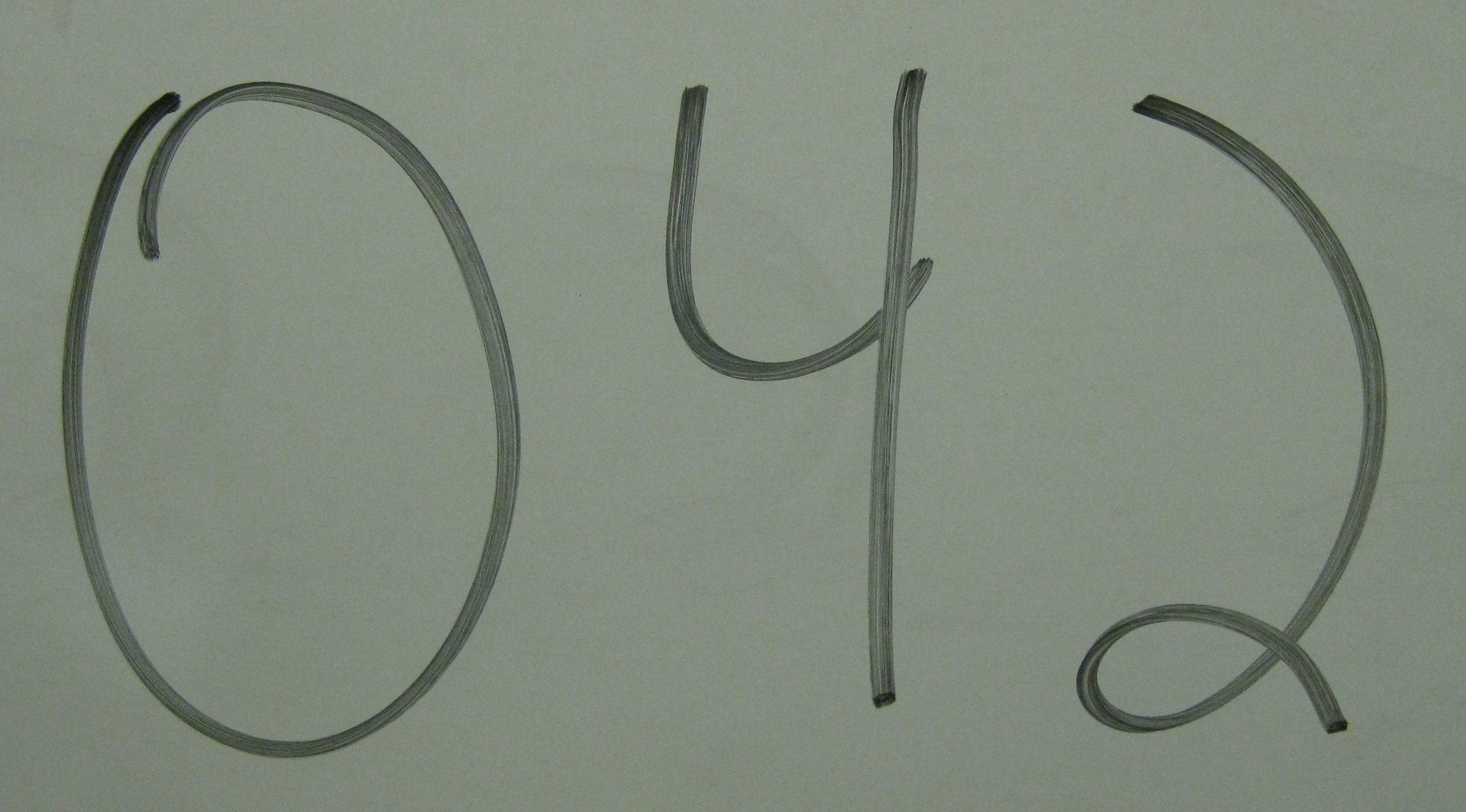}
\end{center}

Besides eliminating the need for dice, this allowed the warden to add more variety to the game; the prisoner's goal could now be \textbf{any} string of nonnegative integers. For example, a game could be played with three numbers, and the prisoner's goal could be to reach the position 314.

Given a particular goal position, from which positions can the prisoner eventually reach the goal, assuming optimal play? We have the following result:

\begin{theorem}
  Let $g_1 g_2 \ldots g_n$ be the goal position, where each $g_i$ is a nonnegative integer. The prisoner can reach the goal position from $x_1 x_2 \ldots x_n$ if and only if the following is true: there exists a rotation $y_1 y_2 \ldots y_n=x_k \ldots x_n x_1 \ldots x_{k-1}$ of the current position, such that $y_i\le g_i$ for all $i$.
\end{theorem}

For example, if 314 is the goal, then the prisoner can force a win from the position 042. This is true because we can rotate 042 to produce 204, and each digit in that string is at most the corresponding digit in the goal string ($2\le 3$, $0\le 1$, $4\le 4$). However, if 314 is the goal, then the prisoner can't force a win from 402; none of the possible rotations (402, 024, and 240) will work.

\begin{proof}
  Assume first that there is no rotation $y_1 y_2 \ldots y_n$ of the current position such that $y_i\le g_i$ for all $i$. That is, for any such rotation, $y_i>g_i$ for some $i$. In this case, the warden can keep the game going indefinitely, simply by never reducing any number. If the warden does that, then all numbers will either stay the same or increase during the game. Thus, at any point in the game, at least one number will be larger than the corresponding number in the goal position; that means the goal will never be reached.
  
  Now assume that there is a rotation $y_1 y_2 \ldots y_n$ of the current position such that $y_i\le g_i$ for all $i$. The prisoner can then force a win by doing the following: first, the prisoner refuses to increase any numbers until the rotation $y_1 y_2 \ldots y_n$ is the current position. (In the meantime, the warden may have reduced some of the numbers. But even if that happens, we will still have $y_i\le g_i$ for all $i$.) Then, the prisoner follows the following strategy for each sequence of $n$ moves: as long as the warden just passes, the prisoner tries for the win by increasing each number to the corresponding number in the goal position. If the warden ever reduces a number, then the prisoner doesn't increase any more numbers for the rest of the cycle. At the end of the cycle, either the prisoner will have reached the goal, or the new position will be lexicographically smaller than the position at the start of the cycle. But we can't get lexicographically smaller positions forever, so the prisoner will eventually reach the goal.
\end{proof}

As an example: consider the earlier example where the goal position is 314 and the current position is 042. If the warden passes, then the prisoner transfers the 2 without changing it, giving the position 204. The prisoner now attempts to move to 314 over the next three turns. If the warden passes (not a smart move), then the prisoner moves to $4|20$. Now the warden has to pass, and the prisoner moves to $14|2$. The warden dare not pass again, or else the prisoner will move to 314 and win. So perhaps the warden moves to 114. This position is lexicographically smaller than the position at the start of the cycle, 204. The prisoner now repeats this strategy until he inevitably wins. 

We can construct a loop of numbers similar to a de Bruijn sequence, showing the optimal sequence of moves for a particular goal position. (As with the case where all numbers in the goal are the same, there will be exactly one position of each remoteness.) For example, here's the optimal sequence of moves when the goal is 321:

\begin{center}
(321)00010110200211120121220221300301310311320321
\end{center}

Note: if we change the goal to one of its rotations, then the prisoner can win from the same set of positions, but they will appear in a different order. Here's the optimal sequence when the goal is 132:

\begin{center}
(132)00010020110120210220300310321112122130131132
\end{center}

And here's the optimal sequence when the goal is 213:

\begin{center}
(213)00010020030110120131021031112113202203212213
\end{center}

I don't know if this generalization of de Bruijn sequences has appeared in the literature before. It seems likely that there are fast algorithms for finding the locations of particular strings in these sequences, as there are for lexicographically minimal de Bruijn sequences.

\section{The Warden Plays With Multiple Goals}

One thing bothered the warden about his game: under optimal play, the game ``tree'' isn't a tree at all, but a single chain. To make the game more interesting (and harder to analyze), the warden started playing games with multiple goals. Here is an example for the reader to puzzle out:

\begin{center}
  \includegraphics[height=3cm]{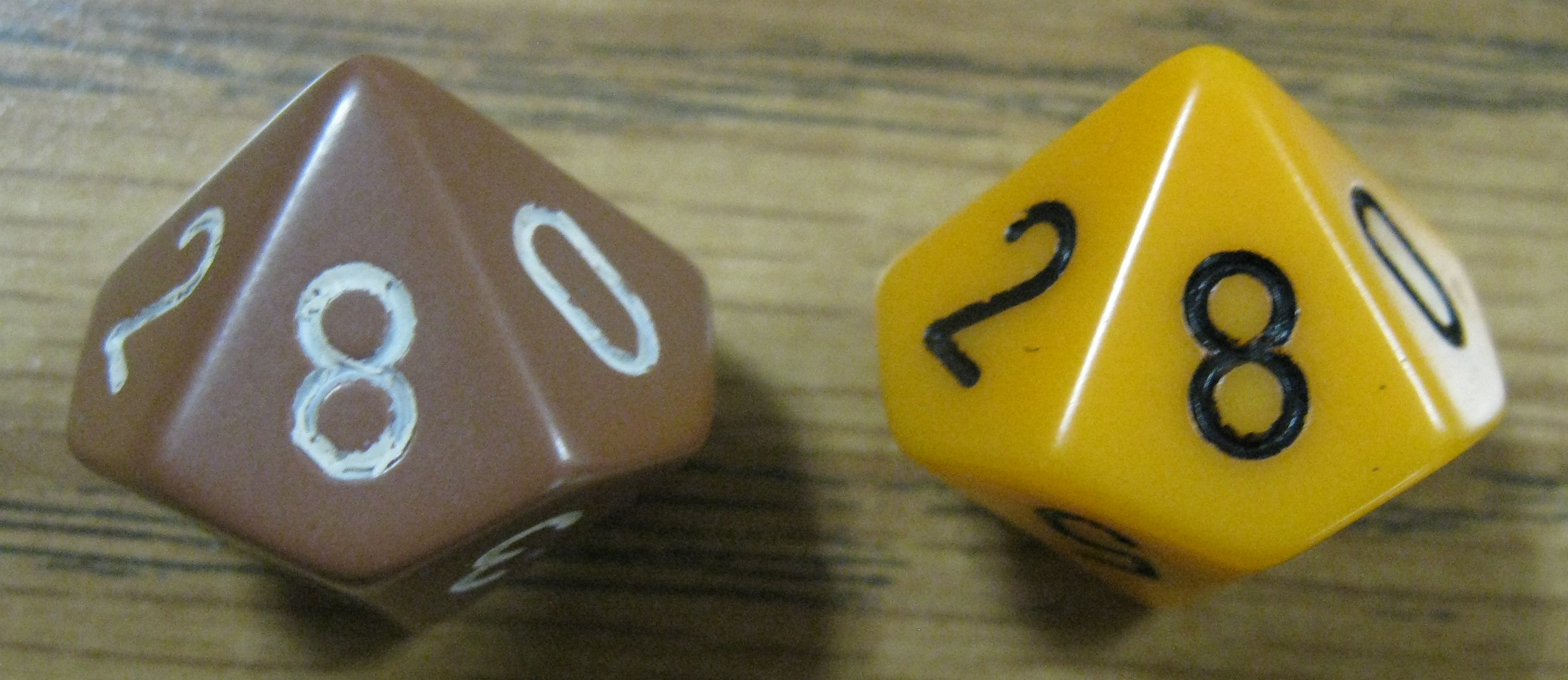}
\end{center}

This game is played with two 10-sided dice, where each die has the numbers from 0 through 9. Both dice show an 8 at the start. You, the prisoner, escape if the two dice together show a two-digit prime number. Leading zeros are allowed; thus, the possible win combinations are 02, 03, 05, 07, 11, 13, and so on.

It may sound like the warden is being generous by giving you many ways to win. But there is a limit to the warden's generosity; he has put a time limit on the game. You will earn your freedom only if the dice are showing a prime number within 19 moves. (Every transfer of a die counts towards those 19 moves, whether you or the warden transfers a die.) If 19 moves pass and no prime number appears, then you lose the game, and the warden will double your prison sentence. Now, the warden is not forcing you to play this game; if you think this task is impossible, then you may refuse to play, and just serve out your original sentence.

As far as I know, there are no quick algorithms for solving variations with multiple goals; I know of no faster way to solve this game than to find the remoteness of each and every position, one at a time. (Unlike the single-goal variations mentioned earlier, there can be multiple positions with the same remoteness.) So, what should you do? Is the warden trying to trick you by giving you an impossible task? Or can you actually force a win within 19 moves? I will leave it to the reader to figure it out.

\end{document}